\documentclass[11pt]{article}      
\usepackage{amsmath,amssymb,amscd, graphicx, verbatim, setspace} 
\usepackage{setspace}
\usepackage{amsthm}
\usepackage{authblk}

\theoremstyle{plain}
\newtheorem{thm}{Theorem}[section]
\newtheorem{lem}[thm]{Lemma}
\newtheorem{cor}[thm]{Corollary}

\title{Improved bounds on maximum sets of letters in sequences with forbidden alternations}
\date{}
\author{Jesse Geneson}
\begin{document}
\maketitle
\begin{abstract}
Let $A_{s,k}(m)$ be the maximum number of distinct letters in any sequence which can be partitioned into $m$ contiguous blocks of pairwise distinct letters, has at least $k$ occurrences of every letter, and has no subsequence forming an alternation of length $s$. Nivasch (2010) proved that $A_{5, 2d+1}(m) = \theta( m \alpha_{d}(m))$ for all fixed $d \geq 2$. We show that $A_{s+1, s}(m) = \binom{m- \lceil \frac{s}{2} \rceil}{\lfloor \frac{s}{2} \rfloor}$ for all $s \geq 2$, $A_{5, 6}(m) = \theta(m \log \log m)$, and $A_{5, 2d+2}(m) = \theta(m \alpha_{d}(m))$ for all fixed $d \geq 3$. 
\end{abstract}

\section{Introduction}
A sequence $s$ \emph{contains} a sequence $u$ if some subsequence of $s$ can be changed into $u$ by a one-to-one renaming of its letters. If $s$ does not contain $u$, then $s$ \emph{avoids} $u$. A sequence $s$ is called \emph{$r$-sparse} if any $r$ consecutive letters in $s$ are distinct. Collections of contiguous distinct letters in $s$ are called \emph{blocks}. 

A \emph{generalized Davenport-Schinzel sequence} is an $r$-sparse sequence avoiding a fixed forbidden sequence with $r$ distinct letters. Bounds on the lengths of generalized Davenport-Schinzel sequences were used to bound the complexity of lower envelopes of sets of polynomials of limited degree \cite{DS}, the complexity of faces in arrangements of arcs \cite{book1}, and the maximum number of edges in simple $k$-quasiplanar graphs \cite{FPS}.

Let $D_{s}(n)$ be the maximum length of any $2$-sparse sequence with $n$ distinct letters which avoids alternations of length $s$. Nivasch \cite{Niv} and Klazar \cite{Kl1} proved $\lim_{n \rightarrow \infty} \frac{D_{5}(n)}{n \alpha(n)} = 2$, such that $\alpha(n)$ denotes the inverse Ackermann function. Agarwal, Sharir, Shor \cite{SAS} and Nivasch \cite {Niv} proved the bounds $D_{s}(n) = n 2^{\frac{1}{t!}\alpha(n)^{t} \pm O(\alpha(n)^{t-1})}$ for even $s \geq 6$ with $t = \frac{s-4}{2}$. Recently Pettie \cite{Pet1} derived sharp bounds on $D_{s}(n)$ for all odd $s$.

Nivasch's bounds on $D_{s}(n)$ were derived using an extremal function which maximizes number of distinct letters instead of length. Let $A_{s,k}(m)$ be the maximum number of distinct letters in any sequence on $m$ blocks avoiding alternations of length $s$ in which every letter occurs at least $k$ times. Clearly $A_{s, k}(m) = 0$ if $m < k$ and $A_{s,k}(m) = \infty$ if $k < s-1$ and $k \leq m$. Nivasch proved that $A_{5, 2d+1}(m) = \theta(m \alpha_{d}(m))$ for each fixed $d \geq 2$. 

\subsection{Ackermann functions}

To define the Ackermann hierarchy let $A_{1}(n) = 2n$ and for $k \geq 2$, $A_{k}(0) = 1$ and $A_{k}(n) = A_{k-1}(A_{k}(n-1))$ for $n \geq 1$. To define the inverse functions let $\alpha_{k}(x) = \min \left\{n : A_{k}(n) \geq x \right\}$ for all $k \geq 1$.

We define the Ackermann function $A(n)$ to be $A_{n}(3)$ as in \cite{Niv}. The inverse Ackermann function $\alpha(n)$ is defined to be $\min \left\{x : A(x) \geq n \right\}$.

\subsection{Formations}

Let an \emph{$(r, s)$-formation} be a concatenation of $s$ permutations of $r$ distinct letters. For example $abcddcbaadbc$ is a $(4,3)$-formation. Define $F_{r,s}(n)$ to be the maximum length of any $r$-sparse sequence with $n$ distinct letters which avoids all $(r,s)$-formations. 

Klazar \cite{Kl} proved $F_{r, 2}(n) = \theta(n)$ and $F_{r, 3}(n) = \theta(n)$ for every $r > 0$. Nivasch proved $F_{r, 4}(n) = \theta(n \alpha(n))$ for $r \geq 2$. Agarwal, Sharir, Shor \cite{SAS} and Nivasch \cite{Niv} showed that $F_{r,s}(n) = n 2^{\frac{1}{t!}\alpha(n)^{t} \pm O(\alpha(n)^{t-1})}$ for all $r \geq 2$ and odd $s \geq 5$ with $t = \frac{s-3}{2}$.

Let $F_{r, s, k}(m)$ be the maximum number of distinct letters in any sequence on $m$ blocks avoiding every $(r, s)$-formation in which every letter occurs at least $k$ times. Clearly $F_{r, s, k}(m) = 0$ if $m < k$ and $F_{r,s,k}(m) = \infty$ if $k < s$ and $k \leq m$. Every $(r, s)$-formation contains an alternation of length $s+1$ for every $r \geq 2$, so $A_{s+1, k}(m) \leq F_{r, s, k}(m)$ for every $r \geq 2$. Nivasch proved for $r \geq 2$ that $F_{r, 4, 2d+1}(m) = \theta(m \alpha_{d}(m))$ for each fixed $d \geq 2$. 

\subsection{Interval chains}\label{prel}

A \emph{$k$-chain on $[1, m]$} is a sequence of $k$ consecutive, disjoint, nonempty intervals of the form $[a_{0}, a_{1}] [a_{1}+1, a_{2}] \ldots [a_{k-1}+1, a_{k}]$ for integers $1 \leq a_{0} \leq a_{1} < \ldots < a_{k} \leq m$. An \emph{$s$-tuple} is a set of $s$ distinct integers. An $s$-tuple \emph{stabs} an interval chain if each element of the $s$-tuple is in a different interval of the chain.

Let $\zeta_{s, k}(m)$ denote the minimum size of a collection of $s$-tuples such that every $k$-chain on $[1,m]$ is stabbed by an $s$-tuple in the collection. Clearly $\zeta_{s, k}(m) = 0$ if $m < k$ and $\zeta_{s, k}(m)$ is undefined if $k < s$ and $k \leq m$. 

Alon \emph{et al.} \cite{AKNSS} showed $\zeta_{s, s}(m) = \binom{m- \lfloor \frac{s}{2} \rfloor}{\lceil \frac{s}{2} \rceil}$ for $s \geq 1$, $\zeta_{3, 4}(m) = \theta(m \log m)$, $\zeta_{3, 5}(m) = \theta(m \log \log m)$, and $\zeta_{3, k} (m) = \theta(m \alpha_{\lfloor \frac{k}{2} \rfloor}(m))$ for $k \geq 6$. 

Let $\eta_{r, s, k}(m)$ denote the maximum size of a collection $X$ of not necessarily distinct $k$-chains on $[1, m]$ so that there do not exist $r$ elements of $X$ all stabbed by the same $s$-tuple. Clearly $\eta_{r, s, k}(m) = 0$ if $m < k$ and $\eta_{r, s, k}(m) = \infty$ if $k < s$ and $k \leq m$. 

\subsection{Our results}

In Section~\ref{formint} we show $\eta_{r, s, k}(m) = F_{r, s+1, k+1}(m+1)$ for all $r \geq 1$ and $1 \leq s \leq k \leq m$. Since $\zeta_{s, k}(m) \geq \eta_{2, s, k}(m)$ for all $1 \leq s \leq k \leq m$, then $A_{s+2, k+1}(m+1) \leq \zeta_{k}^{s}(m)$ for all $1 \leq s \leq k \leq m$. Thus $A_{s+1, s}(m) \leq \binom{m- \lceil \frac{s}{2} \rceil}{\lfloor \frac{s}{2} \rfloor}$ for all $s \geq 2$, $A_{5, 6}(m) = O(m \log \log m)$, and $A_{5, 2d+2}(m) = O(m \alpha_{d}(m))$ for $d \geq 3$.

In Section~\ref{lbounds} we construct alternation-avoiding sequences to prove lower bounds on $A_{s, k}(m)$. We prove that $A_{s+1, s}(m) = \binom{m- \lceil \frac{s}{2} \rceil}{\lfloor \frac{s}{2} \rfloor}$ for all $s \geq 2$. Furthermore we show that $A_{5, 6}(m) = \Omega(m \log \log m)$ and $A_{5, 2d+2}(m) = \Omega(\frac{1}{d} m \alpha_{d}(m))$ for $d \geq 3$. Thus the bounds on $A_{5, d}(m)$ leave a multiplicative gap of $O(d)$ for all $d$. Sundar \cite{Sund} derived bounds of similar order in $m$ on functions related to the Deque conjecture.

\section{Formations and interval chains}\label{formint}
We show $F_{r, s+1, k+1}(m+1) = \eta_{r, s, k}(m)$ for all $r \geq 1$ and $1 \leq s \leq k \leq m$ using maps like those between matrices and sequences in \cite{CK} and \cite{Pet}.

\begin{lem}
\label{equalupper}
$F_{r, s+1, k+1}(m+1) \leq \eta_{r, s, k}(m)$ for all $r \geq 1$ and $1 \leq s \leq k \leq m$.
\end{lem}

\begin{proof}
Let $P$ be a sequence with $F_{r, s+1, k+1}(m+1)$ distinct letters and $m+1$ blocks $1, \ldots, m+1$ such that no subsequence is a concatenation of $s+1$ permutations of $r$ different letters and every letter in $P$ occurs at least $k+1$ times. Construct a collection of $k$-chains on $[1,m]$ by converting each letter in $P$ to a $k$-chain: if the first $k+1$ occurrences of letter $a$ are in blocks $a_{0}, \ldots, a_{k}$, then let $a^{*}$ be the $k$-chain with $i^{th}$ interval $[ a_{i-1}, a_{i}-1 ]$.

Suppose for contradiction that there exist $r$ distinct letters $q_{1}, \ldots, q_{r}$ in $P$ such that $q_{1}^{*}, \ldots, q_{r}^{*}$ are stabbed by the same $s$-tuple $1 \leq j_{1} < \ldots < j_{s} \leq m$. Let $j_{0} = 0$ and $j_{s+1} = m+1$. Then for each $1 \leq i \leq s+1$, $q_{n}$ occurs in some block $b_{n,i}$ such that $b_{n, i} \in [ j_{i-1}+1, j_{i}]$ for every $1 \leq n \leq r$. Hence the letters $q_{1}, \ldots, q_{r}$ make an $(r, s+1)$-formation in $P$, a contradiction.
\end{proof}

\begin{cor}
$A_{s+2, k+1}(m+1) \leq \zeta_{k}^{s}(m)$ for all $1 \leq s \leq k \leq m$.
\end{cor}

The bounds on $\zeta_{s, k}(m)$ in \cite{AKNSS} imply the next corollary.

\begin{cor}
$A_{s+1, s}(m) \leq \binom{m- \lceil \frac{s}{2} \rceil}{\lfloor \frac{s}{2} \rfloor}$ for $s \geq 2$, $A_{5,5}(m) = O(m \log m)$, $A_{5, 6}(m) = O(m \log \log m)$, and $A_{5, k} (m) = O(m \alpha_{\lfloor \frac{k-1}{2} \rfloor}(m))$ for $k \geq 7$.
\end{cor}

To prove $F_{r, s+1, k+1}(m+1) \geq \eta_{r, s, k}(m)$ for all $r \geq 1$ and $1 \leq s \leq k \leq m$, we convert collections of $k$-chains into sequences with a letter corresponding to each $k$-chain. 

\begin{lem}
\label{equallower}
$F_{r, s+1, k+1}(m+1) \geq \eta_{r, s, k}(m)$ for all $r \geq 1$ and $1 \leq s \leq k \leq m$.
\end{lem}

\begin{proof}
Let $X$ be a maximal collection of $k$-chains on $[1, m]$ so that there do not exist $r$ elements of $X$ all stabbed by the same $s$-tuple. To change $X$ into a sequence $P$ create a letter $a$ for every $k$-chain $a^{*}$ in $X$, and put $a$ in every block $i$ such that either $a^{*}$ has an interval with least element $i$ or $a^{*}$ has an interval with greatest element $i-1$. 

Order the letters in blocks starting with the first block and moving to the last. Let $A_{i}$ be the letters in block $i$ which also occur in some block $j < i$ and let $B_{i}$ be the letters which have first occurrence in block $i$. 

All of the letters in $A_{i}$ occur before all of the letters in $B_{i}$. If $a$ and $b$ are in $A_{i}$, then $a$ appears before $b$ in block $i$ if the last occurrence of $a$ before block $i$ is after the last occurrence of $b$ before block $i$. The letters in $B_{i}$ may appear in any order.

$P$ is a sequence on $m+1$ blocks in which each letter occurs $k+1$ times. Suppose for contradiction that there exist $r$ letters $q_{1}, \ldots, q_{r}$ which form an $(r, s+1)$-formation in $P$. List all $(r, s+1)$-formations on the letters $q_{1}, \ldots, q_{r}$ in $P$ lexicographically, so that formation $f$ appears before formation $g$ if there exists some $i \geq 1$ such that the first $i-1$ elements of $f$ and $g$ are the same, but the $i^{th}$ element of $f$ appears before the $i^{th}$ element of $g$ in $P$. 

Let $f_{0}$ be the first $(r, s+1)$-formation on the list and let $\pi_{i}$ (respectively $\rho_{i}$) be the number of the block which contains the last (respectively first) element of the $i^{th}$ permutation in $f_{0}$ for $1 \leq i \leq s+1$. Suppose for contradiction that for some $1 \leq i \leq s$, $\pi_{i} = \rho_{i+1}$. Let $a$ be the last letter of the $i^{th}$ permutation and let $b$ be the first letter of the $(i+1)^{st}$ permutation. 

Then $a$ occurs before $b$ in block $\pi_{i}$ and the $b$ in $\pi_{i}$ is not the first occurrence of $b$ in $P$, so the $a$ in $\pi_{i}$ is not the first occurrence of $a$ in $P$. Otherwise $a$ would appear after $b$ in $\pi_{i}$. Since the $a$ and $b$ in $\pi_{i}$ are not the first occurrences of $a$ and $b$ in $P$, then the last occurrence of $a$ before $\pi_{i}$ must be after the last occurrence of $b$ before $\pi_{i}$. Let $f_{1}$ be the subsequence obtained by deleting the $a$ in $\pi_{i}$ from $f_{0}$ and inserting the last occurrence of $a$ before $\pi_{i}$. Then $f_{1}$ is an $(r, s+1)$-formation and $f_{1}$ occurs before $f_{0}$ on the list. This contradicts the definition of $f_{0}$, so for every $1 \leq i \leq s$, $\pi_{i} < \rho_{i+1}$.

For every $1 \leq j \leq r$ and $1 \leq i \leq s+1$, the letter $q_{j}$ appears in some block between $\rho_{i}$ and $\pi_{i}$ inclusive. Since $\pi_{i} < \rho_{i+1}$ for every $1 \leq i \leq s$, the $s$-tuple $(\pi_{1}, \ldots, \pi_{s})$ stabs each of the interval chains $q_{1}^{*}, \ldots, q_{r}^{*}$, a contradiction. Hence $P$ contains no $(r, s+1)$-formation.
\end{proof}

\section{Lower bounds}\label{lbounds}

In the last section we showed that $A_{s+1, s}(m) \leq \binom{m- \lceil \frac{s}{2} \rceil}{\lfloor \frac{s}{2} \rfloor}$ for $s \geq 2$. The next lemma provides a matching lower bound.

\begin{lem}
$A_{s+1, s}(m) \geq \binom{m- \lceil \frac{s}{2} \rceil}{\lfloor \frac{s}{2} \rfloor}$ for all $s \geq 2$ and $m \geq s+1$.
\end{lem}

\begin{proof}
For every $s \geq 1$ and $m \geq s+1$ we build a sequence $X_{s}(m)$ with $\binom{m- \lceil \frac{s}{2} \rceil}{\lfloor \frac{s}{2} \rfloor}$ distinct letters. First consider the case of even $s \geq 2$. The sequence $X_{s}(m)$ is the concatenation of $m-1$ fans, so that each fan is a palindrome consisting of two blocks of equal length.

First assign letters to each fan without ordering them. Create a letter for every $\frac{s}{2}$-tuple of non-adjacent fans, and put each letter in every fan in its $\frac{s}{2}$-tuple. Then order the letters in each fan starting with the first fan and moving to the last. Let $A_{i}$ be the letters in fan $i$ which occur in some fan $j < i$ and let $B_{i}$ be the letters which have first occurrence in fan $i$.

In the first block of fan $i$ all of the letters in $A_{i}$ occur before all of the letters in $B_{i}$. If $a$ and $b$ are in $A_{i}$, then $a$ occurs before $b$ in the first block of fan $i$ if the last occurrence of $a$ before fan $i$ is after the last occurrence of $b$ before fan $i$. If $a$ and $b$ are in $B_{i}$, then $a$ occurs before $b$ in the first block of fan $i$ if the first fan which contains $a$ without $b$ is before the first fan which contains $b$ without $a$. 

Consider for any distinct letters $x$ and $y$ the maximum alternation contained in the subsequence of $X_{s}(m)$ restricted to $x$ and $y$. Any fans which contain $x$ without $y$ or $y$ without $x$ add at most $1$ to the alternation length. Any fans which contain both $x$ and $y$ add $2$ to the alternation length. If $x$ and $y$ occur together in $i$ fans, then the length of their alternation is at most $(\frac{s}{2}-i)+(\frac{s}{2}-i)+2i = s$.

Every pair of adjacent fans have no letters in common, so every pair of adjacent blocks in different fans can be joined as one block when the $m-1$ fans are concatenated to form $X_{s}(m)$. Thus $X_{s}(m)$ has $m$ blocks and $\binom{m- \frac{s}{2}}{\frac{s}{2}}$ letters, and each letter occurs $s$ times. 

For odd $s \geq 3$, construct $X_{s}(m)$ by adding a block $r$ after $X_{s-1}(m-1)$ containing all of the letters in $X_{s-1}(m-1)$ such that $a$ occurs before $b$ in $r$ if the last occurrence of $a$ in $X_{s-1}(m-1)$ is after the last occurrence of $b$ in $X_{s-1}(m-1)$. Then $X_{s}(m)$ contains no alternation of length $s+1$ since $X_{s-1}(m-1)$ contains no alternation of length $s$. Moreover $X_{s}(m)$ has $m$ blocks and $\binom{m- \frac{s+1}{2}}{\frac{s-1}{2}}$ letters, and each letter occurs $s$ times. 
\end{proof}

The proof of the following lemma is much like the proof in \cite{Niv} that $A_{5, 2d+1}(m) = \Omega(\frac{1}{d} m \alpha_{d}(m))$ for $d \geq 2$. 

\begin{lem}
$A_{5, 6}(m) = \Omega(m \log \log m)$ and $A_{5, 2d+2}(m) = \Omega(\frac{1}{d} m \alpha_{d}(m))$ for $d \geq 3$.
\end{lem}

For all $d, m \geq 1$, we inductively construct sequences $G_{d}(m)$ in which each letter appears $2d+2$ times and no two distinct letters make an alternation of length $5$. This proof uses a different definition of fan: fans will be the concatenation of two palindromes with no letters in common. Each palindrome consists of two blocks of equal length. 

The sequences $G_{1}(m)$ are the concatenation of $m+1$ fans. In each fan the second block of the first palindrome and the first block of the second palindrome make one block together since they are adjacent and have no letters in common. The first palindrome in the first fan and the second palindrome in the last fan are empty. 

There is a letter for every pair of fans and the letter is in both of those fans. The letters with last appearance in fan $i$ are in the first palindrome of fan $i$. They appear in fan $i$'s first palindrome's first block in reverse order of the fans in which they first appear. The letters with first appearance in fan $i$ are in the second palindrome of fan $i$. They appear in fan $i$'s second palindrome's first block in order of the fans in which they last appear. By construction $G_{1}(m)$ contains no alternation of length $5$.

For all $d \geq 1$ the sequence $G_{d}(1)$ consists of $2d+2$ copies of the letter $1$. The first and last copies of $1$ are both special blocks, and there are empty regular blocks before the first $1$ and after the last $1$. 

For $d, m \geq 1$ the blocks in $G_{d}(m)$ containing only first and last occurrences of letters are called special blocks. Let $S_{d}(m)$ be the number of special blocks in $G_{d}(m)$. Every letter has its first and last occurrence in a special block, and each special block in $G_{d}(m)$ has $m$ letters. 

Blocks that are not special are called regular. No regular block in $G_{d}(m)$ has special blocks on both sides, but every special block has regular blocks on both sides. 

The sequence $G_{d}(m)$ for $d, m \geq 2$ is constructed inductively from $G_{d}(m-1)$ and $G_{d-1}(S_{d}(m-1))$. Let $f = S_{d}(m-1)$ and $g = S_{d-1}(f)$. Make $g$ copies $X_{1}, \ldots, X_{g}$ of $G_{d}(m-1)$ and one copy $Y$ of $G_{d-1}(f)$, so that no copies of $G_{d}(m-1)$ have any letters in common with $Y$ or each other. 

Let $A_{i}$ be the $i^{th}$ special block of $Y$. If the $l^{th}$ element of $A_{i}$ is the first occurrence of the letter $a$, then insert $aa$ right after the $l^{th}$ special block of $X_{i}$. If the $l^{th}$ element of $A_{i}$ is the last occurrence of $a$, then insert $aa$ right before the $l^{th}$ special block of $X_{i}$. Replace $A_{i}$ in $Y$ by the modified $X_{i}$ for every $i$. The resulting sequence is $G_{d}(m)$. 

\begin{lem}
For all $d$ and $m$, $G_{d}(m)$ avoids $ababa$.
\end{lem}

\begin{proof}
Given that the alternations in $G_{1}(m)$ have length at most $4$ for all $m \geq 1$, then the rest of the proof is the same as the proof in \cite{Niv} that $Z_{d}(m)$ avoids $a b a b a$.
\end{proof}

Let $L_{d}(m)$ be the length of $G_{d}(m)$. Observe that $L_{d}(m) = (d+1)m S_{d}(m)$ since each letter in $G_{d}(m)$ occurs $2d+2$ times, twice in special blocks, and each special block has $m$ letters.

Define $N_{d}(m)$ as the number of distinct letters in $G_{d}(m)$ and $M_{d}(m)$ as the number of blocks in $G_{d}(m)$. Also let $X_{d}(m) = \frac{M_{d}(m)}{S_{d}(m)}$ and $V_{d}(m) = \frac{L_{d}(m)}{M_{d}(m)}$. We bound $X_{d}(m)$ and $V_{d}(m)$ as in \cite{Niv}. 

\begin{lem}\label{vd}
For all $m, d \geq 1$, $X_{d}(m) \leq 2d+2$ and $V_{d}(m) \geq \frac{m}{2}$.
\end{lem}

\begin{proof}
By construction $S_{1}(m) = m+1$ for $m \geq 1$, $S_{d}(1) = 2$ for $d \geq 2$, and $S_{d}(m) = S_{d}(m-1)S_{d-1}(S_{d}(m-1))$ for $d, m \geq 2$. Furthermore $M_{1}(m) = 3m+3$, $M_{d}(1) = 2d+4$ for $d \geq 2$, and $M_{d}(m) = M_{d}(m-1)S_{d-1}(S_{d}(m-1))+M_{d-1}(S_{d}(m-1))-S_{d-1}(S_{d}(m-1))$ for $d, m \geq 2$.

Thus $S_{2}(m) = S_{2}(m-1)(S_{2}(m-1)+1) \leq 2 (S_{2}(m-1))^{2}$ and $S_{2}(1) = 2$. Since $S'(m) = 2^{2^{m}-1}$ satisfies the recurrence $S'(1) = 2$ and $S'(m+1) = 2 (S'(m))^{2}$, then $2^{2^{m-1}} \leq S_{2}(m) \leq 2^{2^{m}-1}$. For $d \geq 2$, $S_{d}(2) = 2S_{d-1}(2)$ and $S_{1}(2) = 3$. So $S_{d}(2) = 3 \times 2^{d-1}$.

For $d \geq 2$, $M_{d}(2) = (2d+3)(3 \times 2^{d-2})+M_{d-1}(2)$ and $M_{1}(2) = 9$. Hence $M_{d}(2) = (6d+3)2^{d-1}$.

Then $X_{1}(m) = 3$ for all $m$, $X_{d}(1) = d+2$ for all $d \geq 2$, and $X_{d}(2) = 2d+1$ for all $d \geq 2$. For $d, m \geq 2$, $X_{d}(m) = X_{d}(m-1)+\frac{X_{d-1}(S_{d}(m-1))-1}{S_{d}(m-1)}$.

We prove by induction on $d$ that $X_{d}(m) \leq 2d+2$ for all $m, d \geq 1$. Observe that the inequality holds for $X_{1}(m)$, $X_{d}(1)$, and $X_{d}(2)$ for all $m, d$.

Fix $d$ and suppose $X_{d-1}(m) \leq 2d$ for all $m$. Then $X_{d}(m) \leq X_{d}(m-1)+\frac{2d-1}{S_{d}(m-1)}$. Hence $X_{d}(m) \leq X_{d}(2) + (2d-1) \sum_{n = 2}^{\infty} S_{d}(n)^{-1} = 2d+1+(2d-1) \sum_{n = 2}^{\infty} S_{d}(n)^{-1}$.

Since $S_{d}(m) \geq 2 S_{d}(m-1)$ for all $d, m \geq 2$, then $\sum_{n = 2}^{\infty} S_{d}(n)^{-1} \leq 2S_{d}(2)^{-1} = \frac{1}{3 \times 2^{d-2}} \leq \frac{1}{2d-1}$ for all $d \geq 2$, so $X_{d}(m) \leq 2d+2$. Hence $V_{d}(m) = \frac{L_{d}(m)}{M_{d}(m)} = \frac{(d+1)m S_{d}(m)}{M_{d}(m)} \geq \frac{m}{2}$. 
\end{proof}

The following analysis demonstrates the lower bounds on $A_{5, 2d+2}(m)$ for each $d \geq 2$. If $d = 2$, let $m_{i} = M_{2}(i)$ and $n_{i} = N_{2}(i)$. Then $m_{i} = X_{2}(i)S_{2}(i) \leq 6 S_{2}(i) \leq 6 (2^{2^{i}-1}) \leq 2^{2^{i}+2}$ for $i \geq 1$. Then $i = \Omega(\log \log m_{i})$, so $n_{i} = \frac{L_{2}(i)}{6} = \frac{V_{2}(i) M_{2}(i)}{6} \geq \frac{i M_{2}(i)}{12} = \Omega(m_{i} \log \log m_{i})$.

We use interpolation to extend the bound from $m_{i}$ to $m$. Let $i$ and $t$ satisfy $m_{i} \leq m < m_{i+1}$ and $t = \lfloor \frac{m}{m_{i}} \rfloor$. Concatenate $t$ copies of $G_{2}(i)$ with no letters in common for a total of at least $\lfloor \frac{m}{m_{i}}\rfloor n_{i} = \Omega(m \log \log m)$ letters. Hence $A_{5, 6}(m) = \Omega(m \log \log m)$.

We prove $S_{3}(m) \leq A_{3}(2m)$ following the method of \cite{Niv}. Since $S_{3}(m) = S_{3}(m-1)S_{2}(S_{3}(m-1)) \leq S_{2}(S_{3}(m-1))^{2} \leq 2^{2^{S_{3}(m-1)+1}-2}$, then let $F(m) = 2^{2^{m+1}-2}$ and $G(m) = 2^{2^{m}}$. Then $2F(m) = 2^{2^{m+1}-1} \leq 2^{2^{2m}} = G(2m)$ for every $m \geq 0$. Thus $S_{3}(m) \leq F^{(m-1)}(S_{3}(1)) < 2F^{(m-1)}(S_{3}(1)) \leq G^{(m-1)}(2S_{3}(1)) = A_{3}(2m)$. 

Let $m_{i} = M_{3}(i)$ and $n_{i} = N_{3}(i)$. Therefore $m_{i} = X_{3}(i)S_{3}(i) \leq 8 S_{3}(i) \leq A_{3}(2i+2)$ for $i \geq 1$. So $i = \Omega(\alpha_{3}(m_{i}))$ and $n_{i} = N_{3}(i) = \frac{L_{3}(i)}{8} = \frac{V_{3}(i)M_{3}(i)}{8} \geq \frac{i M_{3}(i)}{16} = \Omega(m_{i} \alpha_{3}(m_{i}))$. Then $A_{5, 8}(m) = \Omega(m \alpha_{3}(m))$ by interpolation.

For each $d\geq 4$ we prove $S_{d}(m) \leq A_{d}(m+2)$ by induction on $d$. Since $S_{4}(m) = S_{4}(m-1)S_{3}(S_{4}(m-1)) \leq S_{3}(S_{4}(m-1))^{2}$, then let $F(m) = S_{3}(m)^{2}$. Since $4 F(m) \leq A_{3}(4m)$ and $A_{4}(3) > 4 S_{4}(1)$, then $S_{4}(m) \leq F^{(m-1)}(S_{4}(1)) < 4 F^{(m-1)}(S_{4}(1)) \leq A_{3}^{(m-1)}(4 S_{4}(1)) < A_{4}(m+2)$. 

Fix $d > 4$ and suppose $S_{d-1}(m) \leq A_{d-1}(m+2)$. Define $F(m) = S_{d-1}(m)^{2}$. Since $4 F(m) \leq A_{d-1}(4m)$ and $A_{d}(3) > 4 S_{d}(1)$, then $S_{d}(m) \leq F^{(m-1)}(S_{d}(1)) < 4 F^{(m-1)}(S_{d}(1)) \leq A_{d-1}^{(m-1)}(4 S_{d}(1)) < A_{d}(m+2)$. 

Fix $d \geq 4$. Let $m_{i} = M_{d}(i)$ and $n_{i} = N_{d}(i)$. Then $m_{i} = X_{d}(i)S_{d}(i) \leq (2d+2)S_{d}(i) \leq (2d+2)A_{d}(i+2) \leq A_{d}(i+3)$. Then $i \geq \alpha_{d}(m_{i})-3$, so $n_{i} = \frac{L_{d}(i)}{2d+2} = \frac{V_{d}(i) M_{d}(i)}{2d+2} \geq \frac{i M_{d}(i)}{4d+4} = \Omega(\frac{1}{d} m_{i} \alpha_{d}(m_{i}))$. By interpolation $A_{5, 2d+2}(m) = \Omega(\frac{1}{d} m \alpha_{d}(m))$ for $d \geq 4$. 

\begin{cor}
If $r \geq 2$, then $F_{r, 4, 6}(m) = \eta_{r, 3, 5}(m) = \Omega(m \log \log m)$ and $F_{r, 4, 2d+2}(m) = \eta_{r, 3, 2d+1}(m) = \Omega(\frac{1}{d} m \alpha_{d}(m))$ for $d \geq 3$. 
\end{cor}

\section{Acknowledgments}
This research was supported by an NSF graduate research fellowship. The author thanks Peter Shor for helpful comments on this paper and for improving the bounds on $S_{2}(m)$ in Lemma~\ref{vd}.

\end{document}